\documentclass[a4paper,centertags,fleqn,12pt,reqno]{amsart}
\usepackage{amsfonts}
\usepackage{amssymb,amsmath,latexsym}
\usepackage[left=2cm]{geometry}
\usepackage{amssymb,amsmath,latexsym}
\usepackage{cite}

\newtheorem{thm}{Theorem}
\newtheorem{lem}{Lemma}
\newtheorem{cor}{Corollary}
\newtheorem{remark}{Remark}
\theoremstyle{definition}
\newtheorem{dfn}{Definition}
\theoremstyle{remark} \numberwithin{equation}{section}
\allowdisplaybreaks

\begin{document}
\title[Fekete-Szeg\"{o} problem]{On the Fekete-Szeg\"{o} problem associated
with Libera type close-to-convex functions}
\author{Serap BULUT}
\address{Kocaeli University, Faculty of Aviation and Space Sciences,
Arslanbey Campus, 41285 Kartepe-Kocaeli, TURKEY}
\email{serap.bulut@kocaeli.edu.tr}

\begin{abstract}
The main purpose of this paper is to introduce a new comprehensive subclass
of analytic close-to-convex functions and derive Fekete-Szeg\"{o}
inequalities for functions belonging to this new class by using a different
way. Various known special cases of our results are also pointed out.
\end{abstract}

\maketitle

\section{Introduction}

Let $\mathcal{A}$ denote the class of functions of the form%
\begin{equation}
f(z)=z+\sum_{k=2}^{\infty }a_{k}z^{k}  \label{1.1}
\end{equation}%
which are analytic in the open unit disk $\mathbb{U}=\left\{ z\in \mathbb{C}%
:\left\vert z\right\vert <1\right\} .$ Also let $\mathcal{S}$ denote the
subclass of $\mathcal{A}$ consisting of univalent functions in $\mathbb{U}.$

For $f\in \mathcal{S}$ given by $\left( \ref{1.1}\right) $, Fekete and Szeg%
\"{o} \cite{FS} proved a noticeable result that%
\begin{equation}
\left\vert a_{3}-\mu a_{2}^{2}\right\vert \leq \left\{
\begin{array}{lll}
3-4\mu & , & \mu \leq 0 \\
1+2\exp \left( \frac{-2\mu }{1-\mu }\right) & , & 0\leq \mu \leq 1 \\
4\mu -3 & , & \mu \geq 1%
\end{array}%
\right.  \label{1.a}
\end{equation}%
holds. The result is sharp in the sense that for each $\mu $ there is a
function in the class under consideration for which equality holds.

The coefficient functional%
\begin{equation*}
\phi _{\mu }\left( f\right) =a_{3}-\mu a_{2}^{2}=\frac{1}{6}\left( f^{\prime
\prime \prime }\left( 0\right) -\frac{3\mu }{2}\left( f^{\prime \prime
}\left( 0\right) \right) ^{2}\right)
\end{equation*}%
on $f\in \mathcal{A}$ represents various geometric quantities as well as in
the sense that this behaves well with respect to the rotation, namely%
\begin{equation*}
\phi _{\mu }\left( e^{-i\theta }f\left( e^{i\theta }z\right) \right)
=e^{2i\theta }\phi _{\mu }\left( f\right) \quad \left( \theta \in \mathbb{R}%
\right) .
\end{equation*}

In fact, other than the simplest case when%
\begin{equation*}
\phi _{0}\left( f\right) =a_{3},
\end{equation*}%
we have several important ones. For example,%
\begin{equation*}
\phi _{1}\left( f\right) =a_{3}-a_{2}^{2}
\end{equation*}%
represents $S_{f}\left( 0\right) /6,$ where $S_{f}$ denotes the Schwarzian
derivative%
\begin{equation*}
S_{f}\left( z\right) =\left( \frac{f^{\prime \prime }\left( z\right) }{%
f^{\prime }\left( z\right) }\right) ^{\prime }-\frac{1}{2}\left( \frac{%
f^{\prime \prime }\left( z\right) }{f^{\prime }\left( z\right) }\right) ^{2}.
\end{equation*}

Thus it is quite natural to ask about inequalities for $\phi _{\mu }$
corresponding to subclasses of $\mathcal{S}.$ This is called Fekete-Szeg\"{o}
problem. Actually, many authors have considered this problem for typical
classes of univalent functions (see, for instance \cite{AT, B1, B2, B3, CKS,
CTU, DT, DT2, KM, K2, L, MM}).

A function $f\in \mathcal{A}$ is said to be starlike of order $\beta $%
\thinspace $\left( 0\leq \beta <1\right) $ if it satisfies the inequality%
\begin{equation*}
\Re \left( \frac{zf^{\prime }(z)}{f(z)}\right) >\beta \qquad \left( z\in
\mathbb{U}\right) .
\end{equation*}%
We denote the class which consists of all functions $f\in \mathcal{A}$ that
are starlike of order $\beta $ by $\mathcal{S}^{\ast }(\beta )$. It is
well-known that $\mathcal{S}^{\ast }(\beta )\subset \mathcal{S}^{\ast }(0)=%
\mathcal{S}^{\ast }\subset \mathcal{S}.$

Let $0\leq \alpha ,\beta <1.$ A function $f\in \mathcal{A}$ is said to be
close-to-convex of order $\alpha $ and type $\beta $ if there exists a
function $g\in \mathcal{S}^{\ast }\left( \beta \right) $ such that the
inequality%
\begin{equation*}
\Re \left( \frac{zf^{\prime }(z)}{g(z)}\right) >\alpha \qquad \left( z\in
\mathbb{U}\right)
\end{equation*}%
holds. We denote the class which consists of all functions $f\in \mathcal{A}$
that are close-to-convex of order $\alpha $ and type $\beta $ by $\mathcal{C}%
(\alpha ,\beta )$. This class is introduced by Libera \cite{L}.

In particular, when $\beta =0$ we have $\mathcal{C}(\alpha ,0)=\mathcal{C}%
(\alpha )$ of close-to-convex functions of order $\alpha $, and also we get $%
\mathcal{C}(0,0)=\mathcal{C}$ of close-to-convex functions introduced by
Kaplan \cite{K}. It is well-known that $\mathcal{S}^{\ast }\subset \mathcal{C%
}\subset \mathcal{S}$.

Keogh and Merkes \cite{KM} stated the Fekete-Szeg\"{o} inequalities for
functions in the classes $\mathcal{S}^{\ast }\left( \beta \right) $ and $%
\mathcal{C}$, respectively, as follows:

\begin{thm}
\label{lem5}For $0\leq \beta <1$, let $f\left( z\right) $ given by $\left( %
\ref{1.1}\right) $ belongs to the function class $\mathcal{S}^{\ast }(\beta
).$ Then for any real number $\mu ,$%
\begin{equation*}
\left\vert a_{3}-\mu a_{2}^{2}\right\vert \leq \left( 1-\beta \right) \max
\left\{ 1,\,\left\vert 3-2\beta -4\mu \left( 1-\beta \right) \right\vert
\right\} .
\end{equation*}
\end{thm}

\begin{thm}
\label{thm.KM}If $f\in \mathcal{C}$ and if $\mu $\ is real, then%
\begin{equation*}
\left\vert a_{3}-\mu a_{2}^{2}\right\vert \leq \left\{
\begin{array}{lll}
3-4\mu & , & \mu \leq \frac{1}{3} \\
&  &  \\
\frac{1}{3}-\frac{4}{9\mu } & , & \frac{1}{3}\leq \mu \leq \frac{2}{3} \\
&  &  \\
1 & , & \frac{2}{3}\leq \mu \leq 1 \\
&  &  \\
4\mu -3 & , & \mu \geq 1%
\end{array}%
\right. .
\end{equation*}%
For each $\mu $, there is a function in $\mathcal{C}$ such that equality
holds.
\end{thm}

Recently, Darus and Thomas \cite{DT} generalized the results of Theorem $\ref%
{thm.KM}$ for functions $f\in \mathcal{C}(\alpha ,\beta )$ as the following.

\begin{thm}
\label{thm.DT}Let $f\in \mathcal{C}(\alpha ,\beta )$ and be given by $\left( %
\ref{1.1}\right) $. Then for $0\leq \alpha ,\beta <1$,%
\begin{equation*}
3\left\vert a_{3}-\mu a_{2}^{2}\right\vert \leq \left\{
\begin{array}{lll}
\left( 3-2\beta \right) \left( 3-2\alpha -\beta \right) -3\mu \left(
2-\alpha -\beta \right) ^{2} & , & \mu \leq \frac{2\left( 1-\beta \right) }{%
3\left( 2-\alpha -\beta \right) } \\
&  &  \\
1-2\alpha +\beta \left( 3-2\beta \right) +\frac{4}{3\mu }\left( 1-\beta
\right) ^{2} & , & \frac{2\left( 1-\beta \right) }{3\left( 2-\alpha -\beta
\right) }\leq \mu \leq \frac{2}{3} \\
&  &  \\
3-2\alpha -\beta & , & \frac{2}{3}\leq \mu \leq \frac{2\left( 2-\beta
\right) \left( 3-2\alpha -\beta \right) }{3\left( 2-\alpha -\beta \right)
^{2}} \\
&  &  \\
\left( 2\beta -3\right) \left( 3-2\alpha -\beta \right) +3\mu \left(
2-\alpha -\beta \right) ^{2} & , & \mu \geq \frac{2\left( 2-\beta \right)
\left( 3-2\alpha -\beta \right) }{3\left( 2-\alpha -\beta \right) ^{2}}%
\end{array}%
\right. .
\end{equation*}%
For each $\mu $, there is a function in $\mathcal{C}(\alpha ,\beta )$ such
that equality holds.
\end{thm}

Al-Abbadi and Darus \cite{AD2} introduced a general subclass $\mathcal{U}%
_{\lambda ,\alpha }^{\beta }$ of close-to-convex functions as follows:

\begin{dfn}
\label{dfn1}For $0\leq \lambda <1$, $0\leq \alpha <1$ and $0\leq \beta <1$,
let the function $f\in \mathcal{A}$ be given by $\left( \ref{1.1}\right) $.
Then the function $f\in \mathcal{U}_{\lambda ,\alpha }^{\beta }$ if and only
if there exist $g\in \mathcal{S}^{\ast }(\beta )$ such that%
\begin{equation*}
\Re \left( \frac{zf^{\prime }\left( z\right) +\lambda z^{2}f^{\prime \prime
}\left( z\right) }{g\left( z\right) }\right) >\alpha \qquad \left( z\in
\mathbb{U}\right) .
\end{equation*}
\end{dfn}

Now we define a more comprehensive class of close-to-convex functions of
order $\alpha $ and type $\beta $:

\begin{dfn}
\label{dfn2}Let $0\leq \alpha <1$ and $0\leq \delta \leq \lambda \leq 1.$We
denote by $\mathcal{K}_{\lambda ,\delta }\left( \alpha ,\beta \right) $ the
class of functions $f\in \mathcal{A}$ satisfying%
\begin{equation}
\Re \left( \frac{zf^{\prime }\left( z\right) +\left( \lambda -\delta
+2\lambda \delta \right) z^{2}f^{\prime \prime }\left( z\right) +\lambda
\delta z^{3}f^{\prime \prime \prime }(z)}{g\left( z\right) }\right) >\alpha
\qquad \left( z\in \mathbb{U}\right) ,  \label{def.1}
\end{equation}%
where $g\in \mathcal{S}^{\ast }(\beta );\;0\leq \beta <1.$
\end{dfn}

\begin{remark}
$(i)$ For $\delta =0,$ the class $\mathcal{K}_{\lambda ,\delta }\left(
\alpha ,\beta \right) $ reduces to the class $\mathcal{U}_{\lambda ,\alpha
}^{\beta }$ defined in Definition $\ref{dfn1}$.

\noindent $(ii)$ For $\delta =0$ and $\alpha =0,$ the class $\mathcal{K}%
_{\lambda ,\delta }\left( \alpha ,\beta \right) $ reduces to the class $%
\mathcal{U}_{\lambda }^{\beta }$ which consists of functions $f\in \mathcal{A%
}$ satisfying%
\begin{equation*}
\Re \left( \frac{zf^{\prime }\left( z\right) +\lambda z^{2}f^{\prime \prime
}\left( z\right) }{g\left( z\right) }\right) >0\qquad \left( z\in \mathbb{U}%
\right) ,
\end{equation*}%
where $g\in \mathcal{S}^{\ast }(\beta );\;0\leq \beta <1;\;0\leq \lambda
\leq 1.$ This class introduced and studied by Al-Abbadi and Darus \cite{AD}.

\noindent $(iii)$ For $\delta =0$ and $\lambda =0,$ the class $\mathcal{K}%
_{\lambda ,\delta }\left( \alpha ,\beta \right) $ reduces to the class $%
\mathcal{C}(\alpha ,\beta ).$
\end{remark}

\section{Preliminary Results}

We denote by $\mathcal{P}$ a class of the analytic functions in $\mathbb{U}$
with%
\begin{equation*}
p(0)=1\qquad \text{and\qquad }\Re \left( p\left( z\right) \right) >0.
\end{equation*}

We shall require the following lemmas.

\begin{lem}
\label{lem1}\cite{D} Let $p\in \mathcal{P}$ with $p\left( z\right)
=1+c_{1}z+c_{2}z^{2}+\cdots .$ Then%
\begin{equation*}
\left\vert c_{n}\right\vert \leq 2\quad \left( n\geq 1\right) .
\end{equation*}
\end{lem}

\begin{lem}
\label{lem2}\cite{MM} Let $p\in \mathcal{P}$ with $p\left( z\right)
=1+c_{1}z+c_{2}z^{2}+\cdots .$ Then for any complex number $\nu $%
\begin{equation*}
\left\vert c_{2}-\nu c_{1}^{2}\right\vert \leq 2\max \left\{ 1,\,\left\vert
2\nu -1\right\vert \right\} ,
\end{equation*}%
and the result is sharp for the functions given by%
\begin{equation*}
p\left( z\right) =\frac{1+z^{2}}{1-z^{2}}\quad \text{and}\quad p\left(
z\right) =\frac{1+z}{1-z}.
\end{equation*}
\end{lem}

\begin{lem}
\label{lem4}\cite{R} Let the function $g$ defined by%
\begin{equation}
g(z)=z+\sum_{k=2}^{\infty }b_{k}z^{k}\qquad \left( z\in \mathbb{U}\right)
\label{1.2}
\end{equation}%
belongs to the function class $\mathcal{S}^{\ast }(\beta )\,\left( 0\leq
\beta <1\right) $. Then we have%
\begin{equation*}
\left\vert b_{2}\right\vert \leq 2\left( 1-\beta \right)
\end{equation*}%
and%
\begin{equation*}
\left\vert b_{3}\right\vert \leq \left( 1-\beta \right) \left( 3-2\beta
\right) .
\end{equation*}
\end{lem}

\begin{lem}
\label{thm1}Let the function $f$ given by $\left( \ref{1.1}\right) $ belongs
to the function class $\mathcal{K}_{\lambda ,\delta }\left( \alpha ,\beta
\right) .$ Then%
\begin{equation}
\left( 1+\lambda -\delta +2\lambda \delta \right) \left\vert
a_{2}\right\vert \leq 2-\alpha -\beta  \label{2.1}
\end{equation}%
and%
\begin{equation}
3\left( 1+2\lambda -2\delta +6\lambda \delta \right) \left\vert
a_{3}\right\vert \leq \left( 3-2\alpha -\beta \right) \left( 3-2\beta
\right) .  \label{2.a}
\end{equation}
\end{lem}

\begin{proof}
Let the function $f\in \mathcal{K}_{\lambda ,\delta }\left( \alpha ,\beta
\right) $ be of the form $\left( \ref{1.1}\right) $. Therefore, there exists
a function $g\in \mathcal{S}^{\ast }(\beta )$, defined in $\left( \ref{1.2}%
\right) $, so that%
\begin{equation*}
\Re \left( \frac{zf^{\prime }\left( z\right) +\left( \lambda -\delta
+2\lambda \delta \right) z^{2}f^{\prime \prime }\left( z\right) +\lambda
\delta z^{3}f^{\prime \prime \prime }(z)}{g\left( z\right) }\right) >\alpha
\qquad \left( z\in \mathbb{U}\right) .
\end{equation*}%
It follows from the above inequality that%
\begin{equation}
\frac{zf^{\prime }\left( z\right) +\left( \lambda -\delta +2\lambda \delta
\right) z^{2}f^{\prime \prime }\left( z\right) +\lambda \delta
z^{3}f^{\prime \prime \prime }(z)}{g\left( z\right) }=\alpha +\left(
1-\alpha \right) p(z),  \label{2.3}
\end{equation}%
with%
\begin{equation}
p(z)=1+c_{1}z+c_{2}z^{2}+\cdots \in \mathcal{P}.  \label{2.2}
\end{equation}%
The equality $\left( \ref{2.2}\right) $ implies the equality%
\begin{equation*}
\frac{z+\sum\limits_{k=2}^{\infty }k\left[ \left( 1-\lambda +\delta \right)
+k\left( \lambda -\delta \right) +k\left( k-1\right) \lambda \delta \right]
a_{k}z^{k}}{z+\sum\limits_{k=2}^{\infty }b_{k}z^{k}}=1+\left( 1-\alpha
\right) \sum_{k=1}^{\infty }c_{k}z^{k}.
\end{equation*}%
Equating coefficients of both sides, we have%
\begin{equation}
2\left( 1+\lambda -\delta +2\lambda \delta \right) a_{2}=b_{2}+\left(
1-\alpha \right) c_{1}  \label{2.4}
\end{equation}%
and%
\begin{equation}
3\left( 1+2\lambda -2\delta +6\lambda \delta \right) a_{3}=b_{3}+\left(
1-\alpha \right) b_{2}c_{1}+\left( 1-\alpha \right) c_{2}.  \label{2.5}
\end{equation}%
Using Lemma $\ref{lem1}$ and Lemma $\ref{lem4}$ in $\left( \ref{2.4}\right) $
and $\left( \ref{2.5}\right) $, we easily get $\left( \ref{2.1}\right) $ and
$\left( \ref{2.a}\right) $, respectively.
\end{proof}

\begin{remark}
In Lemma $\ref{thm1}$, letting $\delta =0$; $\delta =0,\alpha =0$; or $%
\delta =\lambda =0$, we have \cite[Lemma 3]{AD2}, \cite[Lemma 2.3]{AD} and
\cite[Lemma 3]{DT}, respectively.
\end{remark}

\section{Main Results}

In this section, we begin by solving the Fekete-Szeg\"{o} problem for
functions belonging to the class $\mathcal{K}_{\lambda ,\delta }\left(
\alpha ,\beta \right) $ when $\mu \in \mathbb{C}.$

\begin{thm}
\label{thm2}Let $f\left( z\right) $ given by $\left( \ref{1.1}\right) $
belongs to the function class $\mathcal{K}_{\lambda ,\delta }\left( \alpha
,\beta \right) .$ Then, for any complex number $\mu ,$%
\begin{eqnarray*}
3\left( 1+2\lambda -2\delta +6\lambda \delta \right) \left\vert a_{3}-\mu
a_{2}^{2}\right\vert &\leq &\left( 1-\beta \right) \max \left\{
1,\,\left\vert 3-2\beta -\mu \Psi _{\lambda ,\delta }\left( \beta \right)
\right\vert \right\} \\
&&+2\left( 1-\alpha \right) \max \left\{ 1,\,\left\vert 1-\mu \frac{\Psi
_{\lambda ,\delta }\left( \alpha \right) }{2}\right\vert \right\} \\
&&+4\left( 1-\alpha \right) \left( 1-\beta \right) \left\vert 1-\mu \frac{%
\Psi _{\lambda ,\delta }\left( 0\right) }{2}\right\vert ,
\end{eqnarray*}%
where%
\begin{equation*}
\Psi _{\lambda ,\delta }\left( s\right) =\frac{3\left( 1+2\lambda -2\delta
+6\lambda \delta \right) }{\left( 1+\lambda -\delta +2\lambda \delta \right)
^{2}}\left( 1-s\right) .
\end{equation*}
\end{thm}

\begin{proof}
Let the function $f\in \mathcal{K}_{\lambda ,\delta }\left( \alpha ,\beta
\right) $ be of the form $\left( \ref{1.1}\right) $. For the simplicity, we
set%
\begin{equation}
\tau =1+\lambda -\delta +2\lambda \delta ,\qquad \sigma =1+2\lambda -2\delta
+6\lambda \delta .  \label{3.2}
\end{equation}%
From $\left( \ref{2.4}\right) $ and $\left( \ref{2.5}\right) $, we obtain%
\begin{eqnarray}
3\sigma \left( a_{3}-\mu a_{2}^{2}\right) &=&\left( b_{3}-\mu \frac{3\sigma
}{4\tau ^{2}}b_{2}^{2}\right) +\left( 1-\alpha \right) \left( c_{2}-\mu
\frac{3\sigma \left( 1-\alpha \right) }{4\tau ^{2}}c_{1}^{2}\right)  \notag
\\
&&+\left( 1-\alpha \right) \left( 1-\mu \frac{3\sigma }{2\tau ^{2}}\right)
b_{2}c_{1}.  \label{3.1}
\end{eqnarray}%
So we have%
\begin{eqnarray*}
3\sigma \left\vert a_{3}-\mu a_{2}^{2}\right\vert &\leq &\left\vert
b_{3}-\mu \frac{3\sigma }{4\tau ^{2}}b_{2}^{2}\right\vert +\left( 1-\alpha
\right) \left\vert c_{2}-\mu \frac{3\sigma \left( 1-\alpha \right) }{4\tau
^{2}}c_{1}^{2}\right\vert \\
&&+\left( 1-\alpha \right) \left\vert 1-\mu \frac{3\sigma }{2\tau ^{2}}%
\right\vert \left\vert b_{2}\right\vert \left\vert c_{1}\right\vert
\end{eqnarray*}%
Hence, by means of Theorem $\ref{lem5}$ and Lemmas $\ref{lem1}$-$\ref{lem4}$%
, we get desired result.
\end{proof}

Letting $\delta =0$ in Theorem $\ref{thm2}$, we get following consequence.

\begin{cor}
Let $f\left( z\right) $ given by $\left( \ref{1.1}\right) $ belongs to the
function class $\mathcal{U}_{\lambda ,\alpha }^{\beta }.$ Then, for any
complex number $\mu ,$%
\begin{eqnarray*}
3\left( 1+2\lambda \right) \left\vert a_{3}-\mu a_{2}^{2}\right\vert &\leq
&\left( 1-\beta \right) \max \left\{ 1,\,\left\vert 3-2\beta -\mu \Psi
_{\lambda }\left( \beta \right) \right\vert \right\} \\
&&+2\left( 1-\alpha \right) \max \left\{ 1,\,\left\vert 1-\mu \frac{\Psi
_{\lambda }\left( \alpha \right) }{2}\right\vert \right\} \\
&&+4\left( 1-\alpha \right) \left( 1-\beta \right) \left\vert 1-\mu \frac{%
\Psi _{\lambda }\left( 0\right) }{2}\right\vert ,
\end{eqnarray*}%
where%
\begin{equation}
\Psi _{\lambda }\left( s\right) =\frac{3\left( 1+2\lambda \right) }{\left(
1+\lambda \right) ^{2}}\left( 1-s\right) .  \label{psi}
\end{equation}
\end{cor}

Letting $\delta =0$ and $\alpha =0$ in Theorem $\ref{thm2}$, we get
following consequence.

\begin{cor}
Let $f\left( z\right) $ given by $\left( \ref{1.1}\right) $ belongs to the
function class $\mathcal{U}_{\lambda }^{\beta }.$ Then, for any complex
number $\mu ,$%
\begin{eqnarray*}
3\left( 1+2\lambda \right) \left\vert a_{3}-\mu a_{2}^{2}\right\vert &\leq
&\left( 1-\beta \right) \max \left\{ 1,\,\left\vert \left( 3-2\beta \right)
-\mu \Psi _{\lambda }\left( \beta \right) \right\vert \right\} +2\max
\left\{ 1,\,\left\vert 1-\mu \frac{\Psi _{\lambda }\left( \alpha \right) }{2}%
\right\vert \right\} \\
&&+4\left( 1-\beta \right) \left\vert 1-\mu \frac{\Psi _{\lambda }\left(
0\right) }{2}\right\vert ,
\end{eqnarray*}%
where $\Psi _{\lambda }$ is defined by $\left( \ref{psi}\right) $.
\end{cor}

Letting $\delta =0$ and $\lambda =0$ in Theorem $\ref{thm2}$, we get
following consequence.

\begin{cor}
Let $f\left( z\right) $ given by $\left( \ref{1.1}\right) $ belongs to the
function class $\mathcal{C}(\alpha ,\beta ).$ Then, for any complex number $%
\mu ,$%
\begin{eqnarray*}
3\left\vert a_{3}-\mu a_{2}^{2}\right\vert &\leq &\left( 1-\beta \right)
\max \left\{ 1,\,\left\vert 3-2\beta -3\mu \left( 1-\beta \right)
\right\vert \right\} \\
&&+2\left( 1-\alpha \right) \max \left\{ 1,\,\left\vert 1-\mu \frac{3\left(
1-\alpha \right) }{2}\right\vert \right\} \\
&&+2\left( 1-\alpha \right) \left( 1-\beta \right) \left\vert 2-3\mu
\right\vert .
\end{eqnarray*}
\end{cor}

Now we prove our main result when $\mu $ is real.

\begin{thm}
\label{thm3}Let $f\left( z\right) $ given by $\left( \ref{1.1}\right) $
belongs to the function class $\mathcal{K}_{\lambda ,\delta }\left( \alpha
,\beta \right) .$ Then%
\begin{equation}
3\left( 1+2\lambda -2\delta +6\lambda \delta \right) \left\vert a_{3}-\mu
a_{2}^{2}\right\vert \leq \left\{
\begin{array}{l}
\left( 3-2\beta \right) \left( 3-2\alpha -\beta \right) -\mu \frac{3\left(
2-\alpha -\beta \right) ^{2}\left( 1+2\lambda -2\delta +6\lambda \delta
\right) }{\left( 1+\lambda -\delta +2\lambda \delta \right) ^{2}}, \\
\\
\quad \text{\textit{if}}\quad \mu \leq \frac{2\left( 1-\beta \right) \left(
1+\lambda -\delta +2\lambda \delta \right) ^{2}}{3\left( 2-\alpha -\beta
\right) \left( 1+2\lambda -2\delta +6\lambda \delta \right) } \\
\\
1-2\alpha +\beta \left( 3-2\beta \right) +\frac{4\left( 1-\beta \right)
^{2}\left( 1+\lambda -\delta +2\lambda \delta \right) ^{2}}{3\left(
1+2\lambda -2\delta +6\lambda \delta \right) \mu }, \\
\\
\quad \text{\textit{if}}\quad \frac{2\left( 1-\beta \right) \left( 1+\lambda
-\delta +2\lambda \delta \right) ^{2}}{3\left( 2-\alpha -\beta \right)
\left( 1+2\lambda -2\delta +6\lambda \delta \right) }\leq \mu \leq \frac{%
2\left( 1+\lambda -\delta +2\lambda \delta \right) ^{2}}{3\left( 1+2\lambda
-2\delta +6\lambda \delta \right) } \\
\\
3-2\alpha -\beta , \\
\\
\quad \text{\textit{if}}\quad \frac{2\left( 1+\lambda -\delta +2\lambda
\delta \right) ^{2}}{3\left( 1+2\lambda -2\delta +6\lambda \delta \right) }%
\leq \mu \leq \frac{2\left( 2-\beta \right) \left( 3-2\alpha -\beta \right)
\left( 1+\lambda -\delta +2\lambda \delta \right) ^{2}}{3\left( 2-\alpha
-\beta \right) ^{2}\left( 1+2\lambda -2\delta +6\lambda \delta \right) } \\
\\
\left( 2\beta -3\right) \left( 3-2\alpha -\beta \right) +\mu \frac{3\left(
2-\alpha -\beta \right) ^{2}\left( 1+2\lambda -2\delta +6\lambda \delta
\right) }{\left( 1+\lambda -\delta +2\lambda \delta \right) ^{2}}, \\
\\
\quad \text{\textit{if}}\quad \mu \geq \frac{2\left( 2-\beta \right) \left(
3-2\alpha -\beta \right) \left( 1+\lambda -\delta +2\lambda \delta \right)
^{2}}{3\left( 2-\alpha -\beta \right) ^{2}\left( 1+2\lambda -2\delta
+6\lambda \delta \right) }.%
\end{array}%
\right.  \label{4.x}
\end{equation}%
For each $\mu $, there is a function in $\mathcal{K}_{\lambda ,\delta
}\left( \alpha ,\beta \right) $ such that equality holds.
\end{thm}

\begin{proof}
Let $f\in \mathcal{K}_{\lambda ,\delta }\left( \alpha ,\beta \right) $ be
given by $\left( \ref{1.1}\right) $ and let us define the function%
\begin{equation}
F_{\lambda ,\delta }(z)=\left( 1-\lambda +\delta \right) f\left( z\right)
+\left( \lambda -\delta \right) zf^{\prime }\left( z\right) +\lambda \delta
z^{2}f^{\prime \prime }(z).  \label{4.a}
\end{equation}%
Then it is worthy to note that the condition $\left( \ref{def.1}\right) $ is
equal to%
\begin{equation*}
\Re \left( \frac{zF_{\lambda ,\delta }^{\prime }(z)}{g\left( z\right) }%
\right) >\alpha \qquad \left( z\in \mathbb{U}\right) .
\end{equation*}%
Since $g\in \mathcal{S}^{\ast }(\beta )$, by the definition of
close-to-convex function of order $\alpha $ and type $\beta $, we deduce
that $F_{\lambda ,\delta }\in \mathcal{C}(\alpha ,\beta ).$ By the
definition of close-to-convex function class $\mathcal{C}(\alpha ,\beta )$,
there exists two functions $p,q\in \mathcal{P}$ with%
\begin{equation*}
p(z)=1+c_{1}z+c_{2}z^{2}+\cdots
\end{equation*}%
and%
\begin{equation*}
q(z)=1+q_{1}z+q_{2}z^{2}+\cdots
\end{equation*}%
such that%
\begin{equation*}
\frac{zF_{\lambda ,\delta }^{\prime }(z)}{g\left( z\right) }=\alpha +\left(
1-\alpha \right) p(z)\qquad \text{with}\qquad \frac{zg^{\prime }(z)}{g\left(
z\right) }=\beta +\left( 1-\beta \right) q(z).
\end{equation*}%
We assume that the function $F_{\lambda ,\delta }$\ is of the form%
\begin{equation}
F_{\lambda ,\delta }(z)=z+\sum_{k=2}^{\infty }A_{k}z^{k}\qquad \left( z\in
\mathbb{U}\right) .  \label{4.b}
\end{equation}%
By Theorem $\ref{thm.DT}$, $F_{\lambda ,\delta }\in \mathcal{C}(\alpha
,\beta )$ implies that%
\begin{equation}
3\left\vert A_{3}-\rho A_{2}^{2}\right\vert \leq \left\{
\begin{array}{lll}
\left( 3-2\beta \right) \left( 3-2\alpha -\beta \right) -3\rho \left(
2-\alpha -\beta \right) ^{2} & , & \rho \leq \frac{2\left( 1-\beta \right) }{%
3\left( 2-\alpha -\beta \right) } \\
&  &  \\
1-2\alpha +\beta \left( 3-2\beta \right) +\frac{4}{3\rho }\left( 1-\beta
\right) ^{2} & , & \frac{2\left( 1-\beta \right) }{3\left( 2-\alpha -\beta
\right) }\leq \rho \leq \frac{2}{3} \\
&  &  \\
3-2\alpha -\beta  & , & \frac{2}{3}\leq \rho \leq \frac{2\left( 2-\beta
\right) \left( 3-2\alpha -\beta \right) }{3\left( 2-\alpha -\beta \right)
^{2}} \\
&  &  \\
\left( 2\beta -3\right) \left( 3-2\alpha -\beta \right) +3\rho \left(
2-\alpha -\beta \right) ^{2} & , & \rho \geq \frac{2\left( 2-\beta \right)
\left( 3-2\alpha -\beta \right) }{3\left( 2-\alpha -\beta \right) ^{2}}%
\end{array}%
\right. .  \label{4.1}
\end{equation}%
Now equating the coefficients of $\left( \ref{4.a}\right) $ and $\left( \ref%
{4.b}\right) $, we obtain%
\begin{equation*}
A_{2}=\tau a_{2},\qquad A_{3}=\sigma a_{3},
\end{equation*}%
where $\tau $ and $\sigma $ defined by $\left( \ref{3.2}\right) $. Hence we
get from the above equalities that%
\begin{equation*}
\left\vert A_{3}-\rho A_{2}^{2}\right\vert =\left\vert \sigma a_{3}-\rho
\tau ^{2}a_{2}^{2}\right\vert =\sigma \left\vert a_{3}-\rho \frac{\tau ^{2}}{%
\sigma }a_{2}^{2}\right\vert .
\end{equation*}%
Taking $\rho =\mu \sigma /\tau ^{2}$ in $\left( \ref{4.1}\right) $, we get
desired estimate $\left( \ref{4.x}\right) $.

Finally, sharpness of the results in $\left( \ref{4.x}\right) $ is getting by

\noindent (i) in Case 1: upon choosing%
\begin{eqnarray*}
c_{1} &=&c_{2}=2, \\
q_{1} &=&q_{2}=2, \\
b_{2} &=&2\left( 1-\beta \right) ,\quad b_{3}=\left( 3-2\beta \right) \left(
1-\beta \right) ,
\end{eqnarray*}

\noindent (ii) in Case 2: upon choosing%
\begin{eqnarray*}
c_{1} &=&\frac{2\left( 1-\beta \right) \left( 2\tau ^{2}-3\sigma \mu \right)
}{3\left( 1-\alpha \right) \sigma \mu },\quad c_{2}=2, \\
q_{1} &=&q_{2}=2, \\
b_{2} &=&2\left( 1-\beta \right) ,\quad b_{3}=\left( 3-2\beta \right) \left(
1-\beta \right) ,
\end{eqnarray*}

\noindent (iii) in Case 3: upon choosing%
\begin{eqnarray*}
c_{1} &=&0,\quad c_{2}=2 \\
q_{1} &=&0,\quad q_{2}=2, \\
b_{2} &=&0,\quad b_{3}=1-\beta ,
\end{eqnarray*}

\noindent (iv) in Case 4: upon choosing%
\begin{eqnarray*}
c_{1} &=&2i,\quad c_{2}=-2, \\
q_{1} &=&2i,\quad q_{2}=-2, \\
b_{2} &=&2\left( 1-\beta \right) i,\quad b_{3}=\left( 3-2\beta \right)
\left( \beta -1\right) .
\end{eqnarray*}
\end{proof}

\begin{remark}
Note that our method proves easily the theorems given by \cite[Theorem 1]%
{AD2} and \cite[Theorem 3.1]{AD} with a different way.
\end{remark}

\begin{cor}
In Theorem $\ref{thm3}$, letting $\delta =0$; $\delta =0,\alpha =0$; $\delta
=\lambda =0$; or $\delta =\lambda =0,\alpha =\beta =0$, we have \cite[%
Theorem 1]{AD2}, \cite[Theorem 3.1]{AD}, Theorem $\ref{thm.DT}$ and Theorem $%
\ref{thm.KM}$, respectively.
\end{cor}

\end{document}